\documentclass[12pt]{article}
\def\Dj{\hbox{D\kern-.73em\raise.30ex\hbox{-}
		\raise-.30ex\hbox{}}}
\def\dj{\hbox{d\kern-.33em\raise.80ex\hbox{-}
		\raise-.80ex\hbox{\kern-.40em}}}
\usepackage{epsfig}
\usepackage{amsmath,amsthm,amsfonts,amssymb,amscd,cite}
\usepackage{tikz}
\usepackage{pgf}
\usepackage{pgfplots}
\usepackage{tkz-graph}
\usepackage{graphviz}
\usepackage{graphicx}
\graphicspath{{./Figures/}}
\usepackage{float}
\usepackage{multirow}
\usepackage{longtable}
\usepackage{subfig}
\usepackage{adjustbox}
\usepackage{array,makecell}
\usepackage{enumitem} 
\usepackage{mathtools}
\usepackage{mathdots}
\usepackage{changepage}
\allowdisplaybreaks

\newtheorem{theorem}{Theorem}[section]
\newtheorem{lemma}[theorem]{Lemma}
\newtheorem{corollary}[theorem]{Corollary}
\newtheorem{remark}[theorem]{Remark}
\newtheorem{definition}[theorem]{Definition}

\newtheorem{example}[theorem]{Example}

\begin{document}

	\baselineskip=0.30in

	\begin{center}
		{\Large \bf On the eccentric graph of trees}
		
		\vspace{6mm}
		
		{\large \bf Sezer Sorgun $^a$, Esma Elyemani$^a$\,}
		
		\vspace{9mm}
		
		\baselineskip=0.20in
		
		$^a${\it Department of Mathematics,\\
			Nevsehir Hac{\i} Bekta\c{s} Veli University, \\
			Nevsehir 50300, Turkey.\/} \\
		e-mail: {\tt srgnrzs@gmail.com, esmaelyemani@gmail.com}\\[3mm]

	\end{center}
	
	\vspace{6mm}
	
	\baselineskip=0.23in
	
		\begin{abstract}
		 We consider the eccentric graph of a graph $G$, denoted by $ecc(G)$, which has the same vertex set as $G$, and two vertices in the eccentric graph are adjacent iff their distance in $G$ is equal to the eccentricity of one of them. In this paper, we present a fundamental requirement for the isomorphism between $ecc(G)$ and the complement of $G$, and show that the previous necessary condition given in the literature is inadequate. Also we obtain that diameter of $ecc(T)$ is at most $3$ for any tree and get some characterizations of the eccentric graph of trees.

		\bigskip
		
	\end{abstract}

	\section{Introduction}

	Consider a simple graph $G = (V, E)$, where the adjacency between vertices $v_i$ and $v_j$ is indicated by $v_iv_j \in E(G)$ or $v_i\sim v_j$. The distance between two vertices $u$ and $v$ in $G$ is symbolized as $d_{G}(u, v)$ or, more concisely, as $d(u,v)$. It signifies the minimum number of edges required to traverse from $u$ to $v$ in the graph. The eccentricity $ecc(u)$ of a vertex $u$ in $G$ is the maximum distance between $u$ and any vertex $v$ in the graph. In other words, it measures how far apart $v$ is from the farthest vertex in $G$. The diameter and radius of any graph $G$, denoted by $diam(G)$ and $rad(G)$,  are the maximum and minimum eccentricity among all vertices, respectively. If a graph $G$ satisfies the condition $diam(G)=rad(G)=d$ for some value $d$, it is referred to as a $d$-self-centered graph. This means that eccentricities of all vertices in $G$ are the same.

Consider a collection of mutually disjoint nonempty finite sets ${V_1, \ldots, V_n}$ of any graph $G$. The construction of a graph, formed by taking the union of the sets $V_1, \ldots, V_n$, is as follows: for each $i$, the vertices within $V_i$ are either all adjacent to each other (forming a clique) or all non-adjacent to each other (forming a coclique). In the graph, two vertices, one from $V_i$ and another from $V_j$, are adjacent if and only if $i\sim j$ in $G$.

The mixed extension $\mathcal{ME}_G[\pm p_1, \pm p_2, \ldots, \pm p_n]$ (also referred to as the mixed extension of $G$ of type $(\pm p_1, \pm p_2, \ldots, \pm p_n)$) is constructed by extending the vertex set of $G$ and specifying the cliques and cocliques according to the given types. In this notation, $+p_i$ represents a clique of size $p_i$, and $-p_i$ represents a coclique of size $p_i$. If the vertices of $G$ consist of cocliques (cliques), then the extension will be referred to as a coclique (clique) extension. \cite{haemers2019spectral}. 

A tree is a connected undirected graph without any cycles or loops. It consists of vertices (also known as nodes) and edges. In a rooted tree, one vertex is chosen as the root, which serves as the starting point for traversing the tree. Each vertex in the tree has a level, indicating the number of edges in the path from that vertex to the root. The root itself is at level 0. Consider two vertices, denoted as $v$ and $w$, in a rooted tree. We say that vertex $v$ is an ancestor of $w$ if it lies on the unique path from $w$ to the root. Similarly, $w$ is considered a descendant of $v$. A vertex in a rooted tree is called a leaf if it does not have any children, meaning it is at the end of a branch and has no vertices connected below it.

The eccentric graph $ecc(G)$ of any graph $G$ has the same vertex set with $G$ and two vertices in $ecc(G)$ are adjacent iff their distance in $G$ is equal to the eccentricity of one of them. The eccentric graph of any graph can be useful in studying properties of the original graph related to distance and connectivity. The study of eccentric graphs has been motivated by their applications in network analysis, social network analysis and image processing. The eccentric graph of a graph has been first introduced by Akiyama et. al in 1985 \cite{akiyama}. The eccentric graph of some graphs with special type such as unique eccentric point graphs, path graphs, cycle graphs etc. (see \cite{siriram, pal, kaspar}).
 	
One of the fundamental problems in graph theory is to determine when two graphs are isomorphic. In this paper, we focus on the problem of determining when $ecc(G)$ is isomorphic to the complement of $G$. This problem has been investigated by Akiyama et. al \cite{akiyama}, where they gave a necessary condition for $ecc(G)$ to be isomorphic to the complement of $G$. However, we show that this condition is inadequate by presenting an example of a graph $G$ whose eccentric graph is isomorphic to its complement, but does not satisfy the previous necessary condition.	

Motivated by this, we correct the necessary condition for $ecc(G)\cong \overline{G}$. We prove our result by using the necessary condition given in \cite{akiyama}, and show that our condition is stronger than the previous one. We also obtain the structure of $ecc(T)$ with respect to its center(s). Finally, we show that the diameter of the eccentric graph of any tree is at most $3$.
	\vspace*{3mm}
	
	\section{Main Results}
	
	\begin{definition}\cite{akiyama}
		Let $G$ be a graph with  vertex set $V$ and edge set $E$. The eccentric graph of $G$ is denoted by  $ecc(G)= (V, E')$ where $E'= \{(u, v) :\ d_G(u, v) = ecc(u) \ \text{or}\ \ d_G(u, v) = ecc(v) \ \text{for} \ u, v \in V\}$.
		\end{definition}
	
	The eccentric graph of some extremal graphs such as path, complete, cycle etc. is given the following table by using the concept of eccentric graph.
	\begin{table}[H]
		\caption{Eccentric graphs of extremal graphs}
		\hskip-1cm
		\centering
		\begin{tabular}{|l|l|}
			\hline
			\hline
			$G$ &  $ecc(G)$  \\
			\hline
			Complete graph $K_n$ &   Complete Graph $K_n$  \\
			Star graph $S_n$ & Complete graph $K_n$  \\
			Cycle graph $C_{2n}$ & Disjoint union of $K_2$ \\
			Cycle graph $C_{2n-1}$ &   Cycle graph $C_{2n-1}$ \\
			Path graph $P_{2n}$&  Double star \\
			Path graph $P_{2n-1}$&   $S^*_{\frac{n-3}{2},\frac{n-3}{2}}$  (Figure 1).\\
			Complete bipartite graph $K_{n,m}$ & Disjoint union of $K_n$ and $K_m$\\
			\hline
		\end{tabular} 
		\label{tab:EccentricGraphsDatas}
	\end{table}

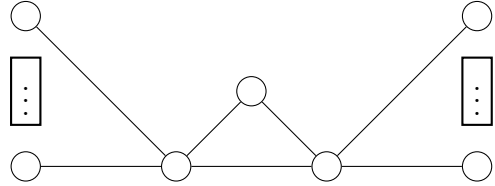
\begin{figure}[h]
	\centering 
	\begin{tikzpicture}
		
		\node[circle, draw] (1) at (0,0) {};
		\node[circle, draw] (2) at (2,0) {};
		\node[circle, draw] (3) at (1,1) {};
		\node[circle, draw] (4) at (4,0) {};
		\node[circle, draw] (5) at (-2,0) {};
		\node[thick, draw] (6) at (-2,1) {$\vdots$};
		\node[circle, draw] (7) at (-2,2) {};
		\node[circle, draw] (8) at (4,2) {};
		\node[thick, draw] (9) at (4,1) {$\vdots$};
		\
		\draw (1) -- (2);
		\draw (1) -- (3);
		\draw (2) -- (3);
		\draw (2) -- (4);
		\draw (1) -- (5);
		\draw (1) -- (7);
		\draw (2) -- (8);
	\end{tikzpicture}
	\caption{$S^{*}_{\frac{n-3}{2},\frac{n-3}{2}}$}
\end{figure}

	\begin{theorem} \cite{akiyama} \label{A}
		 $ecc(G)\cong \overline{G}$ if and only if $S_i=\emptyset$ for $i=1,4,5,6\ldots$  and no two vertices in $S_3$ have a common neighbour , where $S_i=\{u\in V: ecc(u)=i\}$.
	\end{theorem}

 Theorem \ref{A} is given by Akiyama et. al \cite{akiyama}. But the neccessary condition for $ecc(G)\cong \overline{G}$ in the theorem is missing. The $G$ graph in Fig. 2. does not satify one of necessary condition but $ecc(G)\cong \overline{G}$. 
 
\begin{figure}[h]
	\centering 
	\begin{tikzpicture}
		
		\node[circle, draw] (3) at (0,0) {3};
		\node[circle, draw] (2) at (2,0) {2};
		\node[circle, draw] (7) at (1,3) {7};
		\node[circle, draw] (5) at (0,2) {5};
		\node[circle, draw] (6) at (2,2) {6};
		\node[circle, draw] (1) at (0,5) {1};
		\node[circle, draw] (4) at (2,5) {4};
		
		\
		\draw (1) -- (4);
		\draw (1) -- (7);
		\draw (7) -- (5);
		\draw (7) -- (6);
		\draw (5) -- (6);
		\draw (4) -- (7);
		\draw (5) -- (3);
		\draw (6) -- (2);
	\end{tikzpicture}, \ \ \ \ \ \ \begin{tikzpicture}
		
		\node[circle, draw] (2) at (0,0) {2};
		\node[circle, draw] (4) at (2,-1) {4};
		\node[circle, draw] (3) at (2,2) {3};
		\node[circle, draw] (7) at (0,2) {7};
		\node[circle, draw] (6) at (4,2) {6};
		\node[circle, draw] (1) at (4,0) {1};
		\node[circle, draw] (5) at (2,-3) {5};
		
		\
		\draw (7) -- (3);
		\draw (7) -- (2);
		\draw (6) -- (3);
		\draw (6) -- (1);
		\draw (1) -- (2);
		\draw (2) -- (3);
		\draw (2) -- (4);
		\draw (2) -- (5);
		\draw (4) -- (5);
		\draw (4) -- (6);
		\draw (4) -- (3);
		\draw (5) -- (1);
		\draw (1) -- (3);
	\end{tikzpicture} \\
$G$\ \ \ \  \ \ \ \ \ \ \ \ \ \ \ \ \ \ \ $ecc(G)=\overline{G}$
	\caption{In the graph $G$, $S_3=\{1,2,3,4\}$ and its eccentric graph on the right side is also isomorphic to its complement.}
\end{figure}
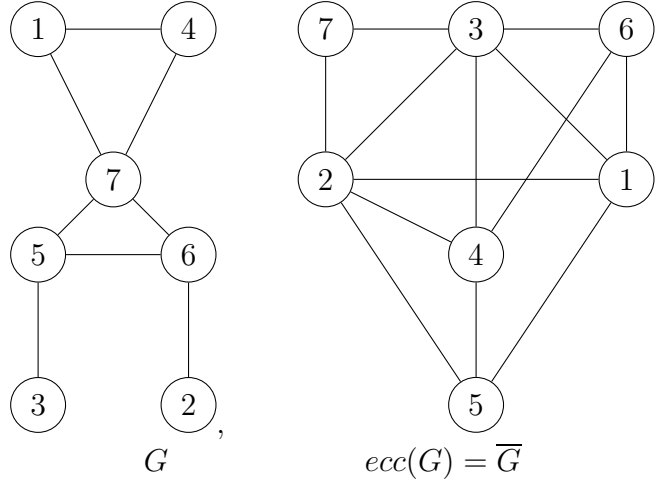

\begin{theorem}
	Let $G=(V,E)$ be a graph of diameter $d$. $ecc(G)$ is isomorphic to the complement of $G$ if and only if  $S_i= \emptyset$ for $i=1,4,5,\ldots$  and $d(u,v)\neq 2$ in $G$ for any $u,v\in S_3$, where  $S_i=\{u\in V: ecc(u)=i\}$.
\end{theorem}

\begin{proof}
	
	Suppose that $ecc(G)$ is isomorphic to the complement of $G$. From Theorem \ref{A}, we have $S_i= \emptyset$ for $i=1,4,5,\ldots$. We will now prove that $d(u,v)\neq 2$ for any $u,v\in S_3$.
	
	Assume that there exist vertices $u$ and $v$ in $S_3$ such that $d(u,v)=2$ in $G$. This implies that $u$ and $v$ are not adjacent in $G$. Moreover, in $ecc(G)$, they are also not adjacent because $d_{G}(u,v)\neq 3$. However, this contradicts the assumption that $ecc(G)$ is isomorphic to the complement of $G$. Therefore, we can conclude that $d(u,v)\neq 2$ in $G$ for any $u,v\in S_3$.
	
	Conversely, let $u$ and $v$ be vertices in $S_3$. Since $d_{G}(u,v)\neq 2$, we have $d_{G}(u,v)=1$ or $d_{G}(u,v)=3$. If $d(u,v)=1$ in $G$, then $u\not\sim v$ in $ecc(G)$. If $d(u,v)=3$ in $G$ (which means $u\not\sim v$ in $G$), then $u\sim v$ in $ecc(G)$. Hence, we can conclude that $u\sim v$ in $G$ if and only if $u\not\sim v$ in $ecc(G)$ for $u,v\in S_3$.
	
	Now, let $u\in S_2$ and $v\in S_3$. If $u\sim v$ in $G$, then $u\nsim v$ in $ecc(T)$ since $d_G(u,v)=1\neq ecc(u)$ (also $ecc(v)$). If $u\nsim v$ in $G$, we have $d(u,v)=2=ecc(u)$ in $G$. This implies that $u\sim v$ in $ecc(G)$. Finally, let $u,v\in S_2$. In this case, $u\sim v$ in $G$ if and only if $u\not\sim v$ in $ecc(G)$ because $S_i=\emptyset$ for $i=1,4,5,\ldots$.
	
	Therefore, in $ecc(G)$, two vertices are adjacent if and only if they are not adjacent in $G$. This implies that $ecc(G)\cong \overline{G}$.

\end{proof}

	\begin{lemma} \cite{jordan}
		Every tree has either one or two centers.
	\end{lemma}

	\begin{lemma} \label{c}
		If $T$ is a tree with two centers,  $ecc(T)$ must be isomorphic to any coclique extension of $P_4$.
	\end{lemma}
\begin{proof}

Let $T$ be a tree with diameter $d$ and radius $r$. Assume that $T$ has two center vertices, denoted as $c_1$ and $c_2$.

Consider the bridge in $T$ connecting $c_1$ and $c_2$, denoted as $c_1c_2$.

We can partite the vertex set of $T$ into two subsets with respect to this bridge as $V_1 =\{v \in V : \text{there exists a path in } T \text{ from } v \text{ to } c_1 \text{ that does not cross } c_1c_2\}$ and $V_2 = \{v \in V : \text{there exists a path in } T \text{ from } v \text{ to } c_2 \text{ that does not cross } c_1c_2\}$. These partitions are disjoint, meaning that no vertex belongs to both $V_1$ and $V_2$. Define $T_{c_1}=\{x\in V_1: d_{T}(x,c_1)< r\}$ and $T_{c_2}=\{x\in V_2: d_{T}(x,c_2)< r\}$. Notice that the remaining vertices are the diametrical vertices in both $V_1$ and $V_2$. Now let $U_1$ and $U_2$ be the set of diametrical vertices in $V_1$ and $V_2$, respectively. Since $d_{T}(x,y)=d$ for $x\in U_{1}, y\in U_2$, by definition of the eccentric graph, they must be adjacent each other. Hence, $ecc(T)$ contains the complete bipartite graph $K_{\vert U_1\vert,\vert U_2\vert}$ as an induced subgraph. Also, for any $x\in T_{c_1}$ and for $y\in U_2$, we get $x\sim y$ in $ecc(T)$ since $d_T(x,y)=ecc(x)$. Similarly, for any $x\in T_{c_2}$ and for $y\in U_1$, we also get $x\sim y$ in $ecc(T)$. Therefore $ecc(T)$ is isomorphic to $\mathcal{ME}_{P_4}[-p_1,-p_2,-p_3,-p_4]$ where $p_1=\vert T_{c_1}\vert$,  $p_2=\vert U_2\vert$, $p_3=\vert U_1\vert$ and  $p_4=\vert T_{c_2}\vert$.   
\end{proof}

\begin{example}
$T'$ in the Fig.3. is a tree with diameter $5$ and radius $3$. We have, $C(T')= \{1,9\}$; $S_{5}=\{4,5,6,7,8,14,15\}=\{4,5,6,7,8\}\cup \{14,15\}$; $T_{1}=\{1,2,3\}$ and $T_{9}=\{9,10, 11,12, 13\}$. Hence $ecc(T')\cong \mathcal{ME}_{P_4}[-3,-2,-5,-5]$\\

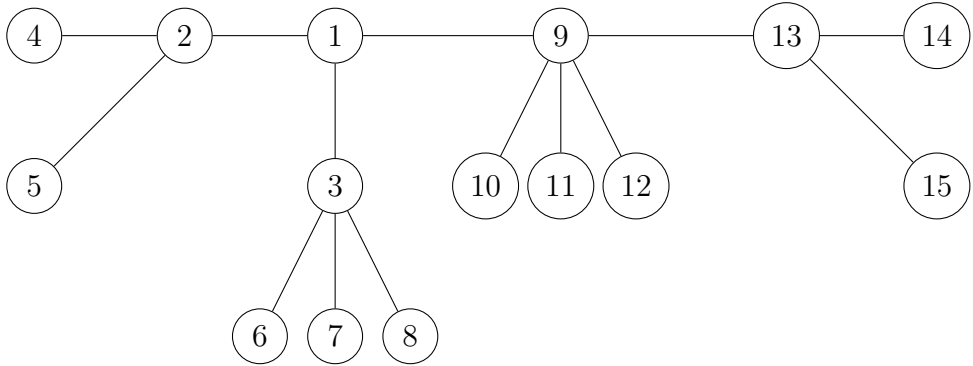
\begin{figure}[h]
	\centering
	\begin{tikzpicture}
		
		\node[circle, draw] (4) at (0,0) {4};
		\node[circle, draw] (2) at (2,0) {2};
		\node[circle, draw] (1) at (4,0) {1};
		\node[circle, draw] (9) at (7,0) {9};
		\node[circle, draw] (13) at (10,0) {13};
		\node[circle, draw] (14) at (12,0) {14};
		\node[circle, draw] (5) at (0,-2) {5};
		\node[circle, draw] (3) at (4,-2) {3};
		\node[circle, draw] (11) at (7,-2) {11};
		\node[circle, draw] (10) at (6,-2) {10};
		\node[circle, draw] (12) at (8,-2) {12};
		\node[circle, draw] (15) at (12,-2) {15};
		\node[circle, draw] (7) at (4,-4) {7};
		\node[circle, draw] (6) at (3,-4) {6};
		\node[circle, draw] (8) at (5,-4) {8};
		\draw (4) -- (2);
		\draw (5) -- (2);
		\draw (2) -- (1);
		\draw (1) -- (9);
		\draw (9) -- (13);
		\draw (13) -- (14);
		\draw (13) -- (15);
		\draw (1) -- (3);
		\draw (3) -- (6);
		\draw (3) -- (7);
		\draw (3) -- (8);
		\draw (9) -- (10);
		\draw (9) -- (11);
		\draw (9) -- (12);

	\end{tikzpicture}
	\caption{A tree $T'$ with 15 vertices}
\end{figure}
\end{example}

\begin{lemma} \label{cc}
Let $T$ be a tree with exactly one center vertex. The graph $ecc(T)$ is a $2$-self centered graph if and only if among the diametrical leaves, there is at least one triple $(x,y,z)$ such that the distance between each pair of vertices $x$,$y$ and $z$ is equal to the diameter of $T$.
\end{lemma}

\begin{proof}
	First, consider $T$ with diameter $d$ and radius $r$  as a rooted tree with $c$ as the center vertex. 
	
	Let $ecc(T)$ is a 2-self centered graph. This means that for any two vertices $x, y \in ecc(T)$, we have $d_{ecc(T)}(x, y) = 2$. Suppose that there is no triple $(x, y, z)$ in $T$ such that the vertices $x$, $y$, and $z$ are diametrical leaves of $T$, and the distance between every pair of vertices $x,y,z$ is equal to the diameter of $T$. Note that $T$ is at least two diametrical pair of leaves.  Let $x$ and $y$ be diametrical vertices in $T$. Now, consider the paths $P$ and $Q$ from the root $c$ to $x$ and $y$ respectively. Let $x \neq u$ and $y \neq w$ be vertices at the same level in $P$ and $Q$, respectively. This implies that $d_{T}(u,c)=d_{T}(w,c)=k<r$. Since $d_{T}(u,w)=d_{T}(u,c)+d_{T}(w,c)=k+k<k+r=ecc(u)$, we have $u\nsim v$ in $ecc(T)$. Notice that $d_{T}(u,y)=ecc(u)$ and $d_{T}(x,w)=ecc(w)$, so we get $u\sim y$ and $x\sim w$, hence there is a path $u-y-x-w$ in $ecc(T)$. Since there is no other diametrical vertex which is at distance $d$ with both $x$ and $y$, hence $u$ and $w$ have no common neighbour in $ecc(T)$. $d(u,w)=3$ in $ecc(T)$. But, this contradicts the assumption that $ecc(T)$ is a 2-self centered graph. Therefore, there must exist at least one triple $(x, y, z)$ in $T$ such that the vertices $x$, $y$ and $z$ are diametrical leaves of $T$ such that the distance between each pair of vertices $x,y,z$ is equal to the diameter of $T$.\\
	
	Conversely, let $x$,$y$ and $z$ be the diametrical leaves such that the distance between each pair of vertices $x,y,z$ is equal to the diameter of $T$. \\
	
	Case 1: Let $x\neq u$ and $x\neq v$ be vertices that lies on path $P$ where $P$ is from the root $c$ to $x$ in $T$. In this case, since $d(u,y)=ecc(u)$ and $d(v,y)=ecc(v)$ in $T$, hence $u\sim y$ and $v\sim y$ in $ecc(T)$. This implies that $d(u,v)=2$ in $ecc(T)$. \\
	
	Case 2: Let $x \neq u$ and $y \neq v$ be vertices that lie on the paths $P$ and $Q$, respectively, where $P$ is the path from the root $c$ to $x$ and $Q$ is the path from the root $c$ to $y$ in $T$. If the vertices $u$ and $v$ are at the same level (i.e., $ecc(u) = ecc(v)$) in $T$, we have $u \nsim v$ in $ecc(T)$ because $d(u,c) = d(v,c) = k < r$ in $T$ and $d(u,v) = 2k < k+r = ecc(u) = ecc(v)$. However, $u$ and $v$ have a common neighbor, denoted by $z$, such that $u \sim z$ and $v \sim z$ in $ecc(T)$, implying that $d(u,v) = 2$. Without loss of generality, assume that $u$ and $v$ are at different levels (i.e., $ecc(u) \neq ecc(v)$) in $T$. Let $k_1 = d(u,c) < d(v,c) = k_2$ in $T$. Then we have $d(u,v) = k_1 + k_2$, but $ecc(u) = k_1 + r$ and $ecc(v) = k_2 + r$. Hence, $u$ and $v$ cannot be adjacent in $ecc(T)$. Similarly to the above argument, $u$ and $v$ have a common neighbor $z$, and thus $d(u,v) = 2$ in $ecc(T)$.\\
	 
	Case 3: Let $u$ and $v$ be the vertices that have a common ancestor (except for $c$) with respect to $x$. In this case, it is obvious that $u$ and $v$ are not adjacent in $ecc(T)$, that is, $u \nsim v$ in $ecc(T)$. Additionally, since $u$ and $v$ share a common ancestor, they also have at least two common neighbors ($y$ and $z$) in $ecc(T)$. Furthermore, both $u$ and $v$ are not adjacent to their respective ancestors. Therefore, we can conclude that $d(u,v) = 2$ in $ecc(T)$. \\
	
	Case 4: Let $u$ and $v$ be any two vertices that have no common ancestor (except for $c$) with respect to each of the vertices $x$, $y$, and $z$. Since $x$, $y$, and $z$ are diametrical leaves, the paths from the root $c$ to $x$, $y$, and $z$ in $T$ are pairwise disjoint. Hence, it is straightforward to observe that $u$ and $v$ are adjacent to each of $x$, $y$, and $z$ in $ecc(T)$. Thus, we can conclude that $d(u,v) = 2$ in $ecc(T)$, as there exists a path of length 2 connecting them through either $x$, $y$, or $z$ in $ecc(T)$.\\
	
	Therefore, we have shown that for every $u$, we have $ecc(u)=2$ in  $ecc(T)$. Thus, we have $diam(ecc(T)) = rad(ecc(T)) = 2$, which implies that $ecc(T)$ is a 2-self centered graph. 
	
\end{proof}

\begin{corollary} \label{ss}
Let $T$ be a tree with exactly one center vertex. If among the diametrical leaves, there does not exist any triple $(x,y,z)$ such that the distance between each pair of vertices $x,y$ and $z$ is equal to the diameter of $T$, $diam(ecc(T))=3$. 
\end{corollary} 
\begin{proof}
	It is obvious from a part of proof of Lemma \ref{cc}.
\end{proof}

\begin{theorem}\label{ccc}
The diameter of $ecc(T)$ is at most 3 for any trees.
\end{theorem}

\begin{proof}
Based on Lemma \ref{c}, it can be inferred that $ecc(T)$ is isomorphic to the coclique extension of $P_4$, implying that $diam(ecc(T))=3$. Taking into account Lemma \ref{cc} and Corollary \ref{ss} , we can conclude directly.
\end{proof}

\begin{remark}
			Theorem \ref{ccc} is applicable only for trees. In general graphs, the diameter of the eccentric graph can exceed 3. For instance, consider the $3 \times 4$ grid graph $G$ shown in Figure 4. In this case, we have $diam(ecc(G)) = 5$, which demonstrates that the diameter can exceed 3 in general graphs.
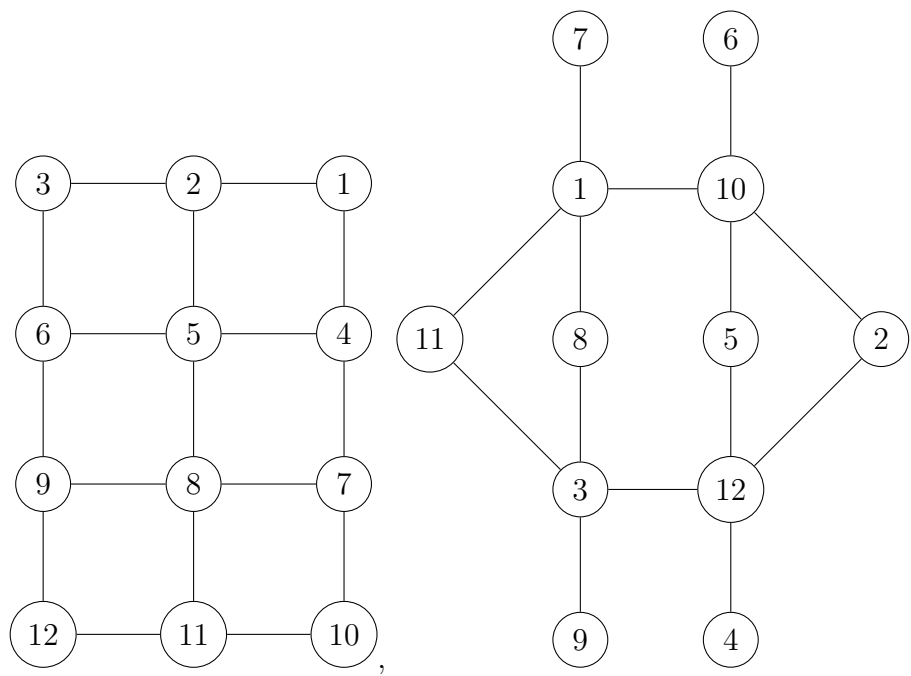
\begin{figure}[h]
	\centering
	\begin{tikzpicture}
		\node[circle, draw] (12) at (0,0) {12};
		\node[circle, draw] (11) at (2,0) {11};
		\node[circle, draw] (10) at (4,0) {10};
		\node[circle, draw] (9) at (0,2) {9};
		\node[circle, draw] (8) at (2,2) {8};
		\node[circle, draw] (7) at (4,2) {7};
		\node[circle, draw] (6) at (0,4) {6};
		\node[circle, draw] (5) at (2,4) {5};
		\node[circle, draw] (4) at (4,4) {4};
		\node[circle, draw] (3) at (0,6) {3};
		\node[circle, draw] (2) at (2,6) {2};
		\node[circle, draw] (1) at (4,6) {1};
		\draw (1) -- (2);
		\draw (2) -- (3);
		\draw (1) -- (4);
		\draw (4) -- (7);
		\draw (7) -- (10);
		\draw (10) -- (11);
		\draw (11) -- (12);
		\draw (12) -- (9);
		\draw (9) -- (6);
		\draw (6) -- (3);
		\draw (6) -- (5);
		\draw (5) -- (4);
		\draw (9) -- (8);
		\draw (8) -- (7);
		\draw (2) -- (5);
		\draw (5) -- (8);
		\draw (8) -- (11);
	\end{tikzpicture}, \begin{tikzpicture}
		
		\node[circle, draw] (6) at (2,8) {6};
		\node[circle, draw] (10) at (2,6) {10};
		\node[circle, draw] (5) at (2,4) {5};
		\node[circle, draw] (12) at (2,2) {12};
		\node[circle, draw] (4) at (2,0) {4};
		\node[circle, draw] (7) at (0,8) {7};
		\node[circle, draw] (1) at (0,6) {1};
		\node[circle, draw] (8) at (0,4) {8};
		\node[circle, draw] (3) at (0,2) {3};
		\node[circle, draw] (9) at (0,0) {9};
		\node[circle, draw] (2) at (4,4) {2};
		\node[circle, draw] (11) at (-2,4) {11};
		\draw (9) -- (3);
		\draw (3) -- (8);
		\draw (8) -- (1);
		\draw (1) -- (7);
		\draw (4) -- (12);
		\draw (12) -- (5);
		\draw (5) -- (10);
		\draw (10) -- (6);
		\draw (1) -- (10);
		\draw (3) -- (12);
		\draw (11) -- (1);
		\draw (11) -- (3);
		\draw (10) -- (2);
		\draw (12) -- (2);
	\end{tikzpicture}
	\caption{$3\times 4$ grid graph and its eccentric graph}
	 
\end{figure}

\end{remark}
It is clear that the diameter of $ecc(G)$  can not be greater than the diameter of $G$ since  $d_{ecc(G)}(u,v)\leq diam(G)$ for any vertices $u$ and $v$. Finally, we propose the following problems:\\

\textbf{Problem 1.} Characterize the graphs which have the same diameter with its eccentric graph.\\

\textbf{Problem 2.} Classify the graphs which is isomorphic to its eccentric graph.


\begin{thebibliography}{99}
				\bibitem{akiyama} J. Akiyama, K. Ando, D. Avis, Eccentric Graphs, Discrete Mathematics 56 (1985) 1-6.
				\bibitem{jordan} A. Dharwadker, S. Pirzada, Graph Theory, CreateSpace Ind. Publishing Platform, 2011.
				\bibitem{haemers2019spectral} W. H. Haemers, Spectral characterization of mixed extensions of small graphs, Discrete Mathematics 342 (10) (2019) 2760-2764.
				\bibitem{siriram} S. Srirama, D. Ranganayakulub, N. Sarminc, I. Venkatd, K.G. Subramaniand, On Eccentric Graphs of Unique Eccentric Point Graphs and Diameter Maximal Graphs, App. Math. and Comp. Intel., 3 (1) (2014) 283–291.
				\bibitem{pal} P. Ghosh, A. Pal, Prime Cordial labeling on eccentric graph of cycle and path, Advanced Modeling and Optimization, 19 (1) (2017) 133-139.
				\bibitem{kaspar} B. Gayathri, S. Kaspar, Domination parameters of some particular classes of eccentric Graphs,  Int. J. of Appl. Engineering Research, 11 (1) (2016) 584-588.
			\end{thebibliography}
\end{document}